\documentclass[12pt,twoside, a4paper]{article}
\usepackage{lmodern}
\usepackage{latexsym}
\usepackage[english]{babel}
\usepackage{mathrsfs}
\usepackage{ifthen}
\usepackage{url}
\usepackage{epsfig}
\usepackage{amsbsy}
\usepackage{extarrows}
\usepackage{amsfonts}
\usepackage{amsmath}
\usepackage{bbm}
\usepackage{amssymb}
\usepackage{amsthm}
\usepackage{amsxtra}
\usepackage{bbm}
\usepackage{color}
\usepackage{stmaryrd}
\usepackage[colorlinks=true]{hyperref}

\usepackage[margin=2cm]{geometry}

\usepackage{verbatim}

\usepackage{hyperref} % HAS TO BE THE LAST PACKAGE TO BE LOADED! - provides links inside the document!!!

\numberwithin{equation}{section}
\numberwithin{figure}{section}

\newtheoremstyle{thm-style-oskari}
{7pt}      % Space above
{7pt}      % Space below
{\itshape} % Body font
{}         % Indent amount (empty = no indent, \parindent = para indent)
{\scshape} % Thm head font
{.}        % Punctuation after thm head
{.5em}     % Space after thm head: " " = normal interword space; 
{}         % Thm head spec (can be left empty, meaning `normal')

\theoremstyle{thm-style-oskari}
    \newtheorem{theorem}{Theorem}[section]
    \newtheorem{proposition}[theorem]{Proposition}
    \newtheorem{corollary}[theorem]{Corollary}
    \newtheorem{lemma}[theorem]{Lemma}
    \newtheorem{assumption}[theorem]{Assumption} 
    \newtheorem{definition}[theorem]{Definition}% howto make rm-style text inside definitions
    
    \newtheorem{convention}[theorem]{Convention}
    \newtheorem{remark}[theorem]{Remark}

\newcommand{\bels}[2] {
        \begin{equation} \label{#1} \begin{split} 
                #2 
        \end{split} \end{equation}
        }
\newcommand{\bes}[1]{
        \begin{equation*}  \begin{split} 
                #1 
        \end{split} \end{equation*}
        }

\definecolor{olivegreen}{rgb}{0,0.6,0.1}

\newcommand{\bs}[1]{\boldsymbol{\mathrm{#1}}} %bold
\renewcommand{\cal}{\mathcal} 
\newcommand{\scr}{\mathscr} 
 
\newcommand{\ol}[1]{\overline{#1} \!\,} %overline

\newcommand{\eps}{\epsilon}

\renewcommand{\d}{\partial}

\renewcommand{\P}{\mathbb{P}}
\newcommand{\E}{\mathbb{E}}
\newcommand{\R}{\mathbb{R}}
\newcommand{\C}{\mathbb{C}}
\newcommand{\N}{\mathbb{N}}

\newcommand{\ee}{\mathrm{e}} %\newcommand{\me}{\mathrm{e}}
\newcommand{\ii}{\mathrm{i}} %\newcommand{\mi}{\mathrm{i}}
\newcommand{\dd}{\mathrm{d}}

\newcommand{\pbb}[1]{\biggl({#1}\biggr)}

\newcommand{\cbb}[1]{\biggl\{{#1}\biggr\}}

\newcommand{\abs}[1]{\lvert #1 \rvert}

\newcommand{\norm}[1]{\lVert #1 \rVert}
\newcommand{\normb}[1]{\big\lVert #1 \big\rVert}

\newcommand{\normbb}[1]{\bigg\lVert #1 \bigg\rVert}

\newcommand{\avg}[1]{\langle #1 \rangle}

\newcommand{\scalar}[2]{\langle{#1} \mspace{2mu}, {#2}\rangle}

\DeclareMathOperator{\diag}{diag}
\DeclareMathOperator{\tr}{Tr}

\DeclareMathOperator{\supp}{supp}

\DeclareMathOperator{\dist} {dist}                
\DeclareMathOperator*{\spec}{Spec}						%Spectrum

\newcommand{\1} {\mspace{1 mu}}
\newcommand{\2} {\mspace{2 mu}}

\newcommand{\mtwo}[2]
{
\left(
\begin{array}{cc}
#1 
\\
#2
\end{array}
\right)
}
\newcommand{\mfour}[4]
{
\left(
\begin{array}{cccc}
#1 
\\
#2
\\
#3
\\
#4
\end{array}
\right)
}
\newcommand{\vtwo}[2]
{
\left(
\begin{array}{c}
#1 
\\
#2
\end{array}
\right)
}

\begin{document}
\title{Power law decay for systems of randomly coupled differential equations}
\author{
\begin{minipage}{0.3\textwidth}
\begin{center}
L\'aszl\'o Erd{\H o}s\footnotemark[1]  \\
\footnotesize {IST Austria}\\
{\url{lerdos@ist.ac.at}}
\end{center}
\end{minipage}
\begin{minipage}{0.3\textwidth}
 \begin{center}
Torben Kr\"uger\footnotemark[1] \\
\footnotesize 
{IST Austria}\\
{\url{torben.krueger@ist.ac.at}}
\end{center}
\end{minipage}
\begin{minipage}{0.3\textwidth}
 \begin{center}
David Renfrew\footnotemark[2] \\
\footnotesize 
{IST Austria}\\
{\url{david.renfrew@ist.ac.at}}
\end{center}
\end{minipage}
}

\footnotetext[1]{Partially supported by ERC Advanced Grant RANMAT No.\ 338804}
\footnotetext[2]{Supported by Austrian Science Fund (FWF): M2080-N35}

\date{\today}
%\date{} % empty argument removes the date!

\maketitle
\thispagestyle{empty} %%% no page numbers!!!

\begin{abstract}
We consider large random matrices $X$ with centered, independent entries but possibly different variances. We compute the normalized trace of $f(X) g(X^*)$ for $f,g$ functions analytic on the spectrum of $X$. We use these results to compute the long time asymptotics for systems of coupled differential equations with random coefficients. We show that when the coupling is critical the norm squared of the solution decays like $t^{-1/2}$.
\end{abstract}

{\bf AMS Subject Classification (2010):} \texttt{60B20}, \texttt{15B52}.

\section{Introduction}

We consider the linear differential equation 
\begin{equation} \label{diffeq}
\partial_t u_t \,=\,  -u_t + g Xu_t
\end{equation}
with $u_t$ an $N$-dimensional vector, $X$ a random matrix with independent entries, and $g$ a positive coupling constant. We treat centered random matrices whose entries may have different variances. This differential equation is solved by exponentiating the random matrix $X$, thus properties of the solution can be computed in terms of functions of $X$. Existing theoretical analysis often assumes matrices have independent and identically distributed (i.i.d.) entries as many results are more readily available in this special case. Equation \eqref{diffeq} has been used for modeling in theoretical neuroscience and mathematical ecology; see, for instance, \cite{SCS,RaAb,ASS,allesina2015predicting,Allesina:2015ux,may1972will}. In both situations, however, the assumption that the entries have identical distributions is not realistic, as it does not allow for spatial information or for different types of species/cells. 

For $g$ large or small the solution exponentially grows or decays, respectively. But there is a critical $g$, where the system exhibits power law decay. We show the decay rate of the solution is universal; it does not depend on the details of the distribution of $X$. In fact, the norm squared of the solution decays like $t^{-1/2}$. In some applications, it is natural  instead to consider $X$ with Hermitian symmetry. In this case, the order of decay is instead $t^{-3/2}$; the slower decay in the non-Hermitian case is a signal of non-normal amplification of transients \cite{nonnormal}. In the neuroscience literature there are many proposals for how certain systems become critically tuned; see, for instance, \cite{HENNEQUIN20141394,lim2013balanced,MacNeil}. 

In the theory of random matrices one is often led to consider the empirical spectral distribution (ESD) of an $N \times N$ matrix, $X$, defined by 
\[ \mu_X := \frac{1}{N} \sum_{i=1}^N \delta_{\lambda_i} \]
where $\lambda_i$ are the eigenvalues of $X$. For random matrices with i.i.d. entries and no Hermitian symmetry, it is well known the ESD converges to a deterministic limit known as the circular law. We refer the reader to \cite{bordenave2012} and references within for an overview. 

For functions, $f$, which are analytic on the spectrum of $X$, $f(X)$ can be defined by Taylor expansion. Then knowledge of the ESD of $X$ allows one to compute traces of functions of $X$ via the formula 
\begin{equation} \label{spectral} \frac{1}{N} \tr f(X) = \int f(x)  \mu_X( \dd x).\end{equation} 

In this non-Hermitian setting, it is of practical and theoretical interest to study functions that depend not just on $X$ but also on $X^*$. Due to the lack of the spectral theorem, no analogue to \eqref{spectral} depending just on the eigenvalues exists for computing traces of general functions of $X$ and $X^*$. In this note, we consider functions $f$ and $g$ which are analytic on a neighborhood of the spectrum of $X$ and $X^*$, respectively, and compute the normalized trace of the product $f(X)g(X^*)$. 

In \cite{CM,MC} random matrices with i.i.d. Gaussian entries are considered, and the expectation of $N^{-1} \tr f(X)g(X^*)$ is computed by diagrammatic techniques with uncontrolled error terms. The computation relies on expressing the expectation of $\tr f(X)g(X^*)$ in terms of the {\it overlap function} of the left and right eigenvectors. This overlap carries delicate information on correlations between eigenvalues. Some of the questions addressed here, in this paper, are considered in the special case of rank 1 variance profiles in \cite{PhysRevE.91.012820}, where techniques similar to those in \cite{CM,MC} were used.

In our analysis, we give mathematically rigorous proofs with high probability error bounds for the random variable $N^{-1} \tr f(X)g(X^*)$, not just its expectation. Rather than treating the eigenvalue overlaps directly, we develop a method for expressing such quantities in terms of a resolvent of the {\it linearization matrix}, evaluated outside the spectrum, where the expressions are simpler and much more stable.

To these matrices we associate a variance profile matrix $S$, whose entry $s_{ij}= \E[|x_{ij}|^2]$ is the variance of the $(i,j)$-entry of $X$. Matrices with general variance profiles were first studied by Girko with the {\it canonical equation of type $K_{25}$} in \cite{girko2012theory}, but the argument was considered incomplete \cite{bai1997}. In the i.i.d. case, a rigorous proof of the convergence of the ESD was given by Bai in \cite{bai1997}, under certain technical assumptions. This work was recently extended in \cite{AEK-circ} and \cite{CHNR} to include local spectral scales and very general variance profiles, respectively. The ESD of such matrices converges to a deterministic measure which is radially symmetric around origin and supported on a disk with radius given by the square root of the spectral radius of the variance profile. 

We use Cauchy?s theorem to reduce the computation of $Tr(f(X)g(X^*)$ to that of computing traces of products of resolvents, $(X-\zeta_1)^{-1}(X^{*}-\ol{\zeta_2})^{-1}$. We introduce a block matrix with an additional parameter whose blocks are linear in $(X-\zeta_i)$. By differentiating in this additional parameter the product of resolvents is expressed in terms of the entries of the resolvent of the linearized matrix. The resolvent of the linearized matrix is then studied with the matrix Dyson equation \cite{AjankiCorrelated}. Our formulas for functions of $X$ and $X^*$ are then applied to compute the long time asymptotics of the norm and autocorrelation of solutions to the differential equations introduced at the beginning of this paper.

\section{Assumptions and Main Results}

Let $X$ be a random $N \times N$ matrix with independent complex entries such that the distribution of the $(i,j)$-entry has zero mean and variance $s_{ij}$:
\[ \E[x_{ij}] = 0, \quad s_{ij} := \E[|x_{ij}|^2] .\]
We call the $N\times N$ matrix $S$ whose $(i,j)$-entry is $s_{ij}$ the {\it variance profile} of $X$. 

\begin{assumption} \label{assum:randmat}
The distribution of the entries of the matrix $X$ satisfies the following:
\begin{enumerate}
\item[(1)] \label{assum:boundvar} (Uniform primitivity) There is a constant $\kappa_1 > 0$ and an integer $K$ such that 
\[    {\left[S^K\right]_{ij}  \geq \frac{\kappa_1}{N},\qquad \left[(S^* S)^K\right]_{ij} \geq \frac{\kappa_1}{N}   } \]
for all $1\leq i,j \leq N$.
\item[(2)]  \label{assum:specrad} (Normalization) 
The spectral radius of $S$  is normalized to one; $\rho(S) = 1$.
\item[(3)]  \label{assum:boundmoments} (Bounded moments) For each $p \in \N$, there exists a $\varphi_p > 0$ such that 
\[ \E[|x_{ij}|^p] \leq \varphi_p N^{-p/2} \]
for all $1\leq i,j \leq N$. 
\end{enumerate}
\end{assumption}

We will consider $\kappa_1, K $, and $\varphi_p$ model parameters. Our estimates will be uniform in all models 
 possessing the same model parameters. In particular, they are uniform in the dimension,~$N$, of the random matrix.

Under Assumption \ref{assum:randmat} the Perron-Frobenius theorem implies that $\rho(S)$ is the largest eigenvalue in magnitude of $S$ and its associated eigenvector has positive entries. We label this eigenvector $v_r$. Similarly, we denote by $v_l$ the Perron-Frobenius eigenvector of $S^*$, i.e.
\[ S v_r = v_r, \qquad S^* v_l = v_l .\]

\begin{remark}
In item (2) the normalization $\rho(S)=1$ can be replaced with any constant independent of $N$. Item (3) can be replaced with $\sqrt{N} x_{ij} $ having finite but sufficiently large number of moments. See Remark \ref{Seigenvector} for comments on item (1).  
\end{remark}

To the variance profile $(s_{ij})$ we associate an operator $\scr{S}$ and its adjoint $\scr{S}^*$  mapping $N \times N$ matrices to $N \times N$ matrices given by
\begin{equation}
\label{defS}
 \scr{S}[T]_{ij}\,:=   \delta_{ij} \sum_{k=1}^N s_{ik} t_{kk}, \quad\qquad
  \scr{S}^*[T]_{ij}\,:= \delta_{ij} \sum_{k=1}^N s_{ki} t_{kk}  .
\end{equation}

Equivalently, one could define $\scr{S}$ and $\scr{S}^*$ by $\scr{S}[T] = \E\2[XTX^*]$ and $\scr{S}^*[T] =\,\E\2[X^*TX]$. Here $\scr{S}^*$ is the adjoint of $\scr{S}$ with respect to the natural scalar product $(A, B)\mapsto \tr(A^*B)$ on $N\times N$ matrices. These operators will appear in the matrix Dyson equation \eqref{eq:MDE}.

We denote by $\tr_N= \frac{1}{N} \tr$ the normalized trace. We use $\langle x, y \rangle = \frac{1}{N} \sum_{i} \ol{x}_i y_i$ to denote the standard normalized inner product on $\C^N$ 
and $\langle x \rangle = \langle \mathbbm{1}, x\rangle$ to denote the average of a vector, where $\mathbbm{1}$ is the constant $N$-dimensional vector with every entry equal to one. Our main theorem computes the limit of the normalized trace of products of analytic functions of $X$ and $X^*$ as $N \to \infty$.

\begin{theorem} \label{analyticfunc}
Let $X$ satisfy Assumption \ref{assum:randmat} and $f,g$ be analytic functions on a neighborhood of the closed disk with radius $1+\mu$ centered at the origin for some $\mu>0$. 
Let $\gamma$ be a positively oriented circle, centered at the origin, with radius $1+\mu$. Then there exists a universal constant $\xi$ such that 
\[  \P\left( \left|  \tr_N f(X) g(X^*) - \left( \frac{1}{2 \pi \ii} \right)^2 \oint_{\gamma} d\zeta_1 \oint_{\ol{\gamma}}  d\ol\zeta_2  \,  f(\zeta_1)g(\overline \zeta_2)  \langle (\zeta_1 \ol\zeta_2 - S )^{-1} \mathbbm{1} \rangle \right|   \leq   \frac{1}{N^{\xi}} \right)  \geq 1 - \frac{C_D}{N^D}\] 
for any $D$ and some positive constant $C_D$. The constant $C_D$ depends on depends on the model parameters, the maximum of $f,g$ on $\gamma$, and $\mu$. 
Here $\ol{\gamma}$ traces the same circle as $\gamma$ but is negatively oriented.
\end{theorem}
Here and in what follows we will typically omit the identity matrix or operator when multiplied by a complex number. Theorem \ref{analyticfunc} is proved after the statement of Theorem \ref{thm:resolventconv}.

\begin{remark} \label{Seigenvector}
It will be clear from the proofs of Theorems \ref{analyticfunc} and \ref{thm:resolventconv} that these results also hold \ if 
the primitivity condition (1) is replaced with the assumption that the matrix $S$ is irreducible, meaning it cannot be brought into block upper-triangular form by conjugation with a permutation matrix, and that there is a $k>0$ such that the entries of the Perron-Frobenius eigenvectors satisfy
\[ (v_l)_i, (v_r)_i > k \quad \text{ for all } 1\leq i \leq N, \]
uniformly in $N$. In this case, $k$ is a model parameter. It is easy to check that this bound holds for $X$ satisfying Assumption~\ref{assum:randmat}.
\end{remark}

\begin{remark}
In the case $s_{ij} =1/N$ for all $i,j$, the integral kernel simplifies to $ \langle (\zeta_1 \ol\zeta_2 - S )^{-1} \mathbbm{1} \rangle = (\zeta_1 \ol\zeta_2 - 1 )^{-1}$, recovering the formula computed in \cite{CM}, \cite{MC} in the Gaussian case. 
\end{remark}

In Section \ref{sec:mainresult}, we use Theorem \ref{analyticfunc} to compute the long time asymptotics of systems of linear ODEs, coupled by a random matrix, when averaged over initial conditions. This system of ODEs is a popular model in theoretical neuroscience; see, for instance, \cite{SCS,RaAb,ASS}.

\begin{theorem} \label{thm:heat}
Let $X$ satisfy Assumption \ref{assum:randmat} and let $u_t \in \C^{N}$ solve the linear ODE
\begin{equation} \label{ODEintial}
\partial_t u_t \,=\,  -u_t + g Xu_t
\end{equation}
with non-Hermitian coefficient matrix $X$, coupling coefficient $0 < g \leq 1$, and initial value $u_0$ distributed uniformly on the $N$-dimensional unit sphere, $\{u : \|u\|=1\} \subset \C^N $. Then there exist universal constants $\xi$ and $\delta$ as well as constants $c,C>0$ such that
\[  \P\left( \left| \E_{u_0} \|u_t\|^2 - \frac{\langle v_l \rangle\langle v_r \rangle}{\langle v_l, v_r \rangle} J_0(2 \ii g t) e^{-2t}  \right| \leq N^{-\xi} + C \ee^{-(2(1 -  g) +cg)t} : \forall t \leq N^{\delta} \right) \geq 1 - \frac{C_D}{N^{D}}\]
for any 
$D>0$ and some constant $C_D$. The function $J_0$ is the Bessel function of the first kind. Here $c,C$ depend only on the model parameters, $\E_{u_0}$ denotes expectation with respect to the initial condition, and $\P$ is the probability with respect to $X$. 

\end{theorem}

\begin{remark}
In particular, in the critical $g=1$ case, using the asymptotics of the Bessel function
\[ 
\lim_{t\to\infty}\lim_{N\to\infty}\E_{u_0} \sqrt{t} \|u_t\|^2 = \frac{\langle v_l \rangle\langle v_r \rangle}{2\sqrt{\pi} \langle v_l, v_r \rangle} . 
\]
where the limit $N \to \infty$ holds in probability with respect to the randomness of $X$. 
\end{remark}

The coefficient $-1$ of the damping term in the differential equation is the negative of the square root of the Perron-Frobenius eigenvalue of $S$. Since the spectral radius of $X$ converges to $\sqrt{\rho(S)}$, the condition $0< g \leq 1$ ensures the real parts of the eigenvalues of $g X$ will be less than or equal to 1, so the differential equation is stable. The long time asymptotics of \eqref{ODEintial} for the i.i.d. model were computed in \cite{CM,MC}. Our results not only give a rigorous proof, but also prove universality. The asymptotics depend on the variance profile only through its spectral radius $\rho(S)$ and Perron-Frobenius eigenvectors.

In Section \ref{sec:mainresult} we also show that as a consequence of the square root behavior at the edge of the spectrum for Hermitian random matrices the long term asymptotics instead decay like $t^{-3/2}$.   

We also compute the autocorrelation function, 
\[ R(\tau) := \E_{B}[\langle u_t, u_{t+\tau}  \rangle ],\] averaged with respect to the Brownian motion and indices, of the solution to
\begin{equation} \label{stationary} \dd u_t \,=\,  (-u_t + gXu_t) \dd t + \dd B_t,\end{equation} with $0<g<1$ and $ B_t$ is a standard Brownian motion, independent of $X$, in stationarity. In the case that all matrix entries have the same variance this was computed in \cite{correlationsMarti}, there the case that the $(i,j)$ and $(j,i)$ entries are correlated was also considered. In the absence of these correlations, we show the autocorrelation exhibits universal decay: to leading order it depends only on $1-g$. 

\begin{theorem} \label{thm:stationary}
Let $X$ satisfy Assumption \ref{assum:randmat}. Then there exist universal constants $\xi$ and $\delta$ as well as $c,C>0$ such that the autocorrelation function of the solution
\begin{equation} \label{ODEstat}  u_t =  \int_{-\infty}^t \ee^{(-1+g X)(t-s)} \, \dd B_s   \end{equation}
to \eqref{stationary} satisfies 
\[ \P\left( \left|R(\tau) - \frac{1}{ 2\sqrt{1 -  g^2 } } \ee^{  - \tau \sqrt{1 -  g^2 }  }    \right| \leq N^{-\xi} + C  \ee^{  - \left( \sqrt{1 -  g^2 }-cg \right) \tau}  : \forall \tau \leq N^{\delta} \right) \geq 1 - \frac{C_D}{N^{D}}
\]
for any  
$D>0$ and some constant $C_D$. Here $c,C$ depends only on the model parameters, and $\P$ is probability with respect to $X$.
\end{theorem}

The computation for general analytic functions in Theorem \ref{analyticfunc} will follow from Cauchy's theorem and Theorem \ref{thm:resolventconv} below concerning the behavior of the trace of the products of two resolvents at spectral parameters $\zeta_1,\zeta_2$ outside the support of the eigenvalues. To quantify the distance from the expected location of the spectrum we introduce the quantity 
\begin{equation}\label{deltazeta} \Delta_\zeta := \min\{|\zeta_1|,|\zeta_2|\}-1, \end{equation}
where we use the shorthand $\zeta := (\zeta_1,\zeta_2) \in \C^2 $. We will always work on the sets of the form
\[ \Xi_{\infty}  := \{ \zeta :   |\zeta_1|,|\zeta_2| \in (1,2)\} \quad \text{or} \quad  \Xi_\delta := \{ \zeta : |\zeta_1|,|\zeta_2|\in (1+N^{-\delta},2)  \} \]
for some $\delta > 0$. The constant 2 can be replaced with any finite constant greater than 1.

\begin{theorem} \label{thm:resolventconv}
Let $X$ satisfy Assumption \ref{assum:randmat}. There exist universal constants $\xi>0$ and $\delta>0$ such that
\[   \P\left(\sup_{\zeta \in \Xi_{\delta} }  \big| \tr_N((X-\zeta_1)^{-1}(X^*-\ol\zeta_2)^{-1}) - \langle (\zeta_1 \ol\zeta_2 - S )^{-1} \mathbbm{1} \rangle )  )
 \big|  \leq \frac{1}{N^{\xi}} \right)\geq 1- \frac{C_D}{N^D} \]
for any $D>0$ and  some constant $C_D$.

\end{theorem}

\begin{proof}[Proof of of Theorem \ref{analyticfunc}]

Corollary \ref{corr:resolventbound} will show that with high probability there is a disk of radius $1+N^{-\delta}$ centered at the origin such that outside the disk the resolvent $(X-\zeta_1)^{-1}$ is bounded. Thus for $N$ sufficiently large, all eigenvalues of $X$ lie inside the
circle $\gamma$ of radius $1+\mu$. We can thus apply Cauchy's theorem, 
\begin{equation}\label{1tr}
 \tr_N[ f(X) g(X^*)]  = \left(\frac{1}{2 \pi \ii} \right)^2 \int_\gamma  \dd \zeta_1 \int_{\ol{\gamma}} \dd \ol{\zeta}_2 \,  f(\zeta_1) g(\ol\zeta_2) \tr_N(X-\zeta_1)^{-1}  (X^*-\ol\zeta_2)^{-1}, 
 \end{equation}
where $\gamma$ is positively oriented and $\ol{\gamma}$ is the same circle oriented negatively. 

Then applying Theorem \ref{thm:resolventconv} yields 
\bes{ \tr_N[ f(X) g(X^*)]  =  \left(\frac{1}{2 \pi \ii} \right)^2  \int_\gamma  \dd \zeta_1 \int_{\ol{\gamma}} \dd \ol{\zeta}_2 \,  f(\zeta_1) g(\ol\zeta_2) \langle (\zeta_1 \ol\zeta_2 - S )^{-1} \mathbbm{1} \rangle    + \epsilon_N, }
where $|\epsilon_N | \lesssim \| f|_\gamma\|_\infty \|g|_\gamma\|_\infty N^{-\xi} $
with probability at least $1 - C_DN^{-D}$ for any $D$ and $N$ sufficiently large.
\end{proof}
We prove Theorem~\ref{thm:resolventconv} at the end of this section, by combining several technical lemmas. Section~\ref{sec:MDE} is devoted to proving these technical lemmas. In Section~\ref{sec:mainresult} we prove Theorem~\ref{thm:heat} and \ref{thm:stationary}. We now outline the proof of Theorem~\ref{thm:resolventconv}, listing the technical lemmas in the reverse order of the actual proof. This better motivates the study of our matrix Dyson equation. However, the actual proofs will avoid any forward referencing to yet unproved lemmas.

Studying polynomials of matrices by instead considering larger block matrices has been successfully implemented to solve a number of problems in random matrix theory; see, for instance, \cite{anderson,haagerup}. Recently this program has been extended to also studying rational functions of random matrices \cite{HMSrealizations}. We will adopt this idea to the present problem.

Furthermore the spectrum of non-Hermitian matrices, in contrast to the Hermitian setting, can be quite unstable and many complex analytic techniques fail. To handle this problem we follow the Hermitization trick of Girko and reduce the problem to studying the spectrum of a family of Hermitian random matrices.

We begin our analysis with a linearization of $(X-\zeta_1)^{-1}(X^*-\ol\zeta_2)^{-1} $, which we then modify in order to compute its limit. Since $(X-\zeta_1)^{-1} (X^*-\ol\zeta_2)^{-1}$ is the (1,1)-entry of the inverse of the block matrix 
\[
 \mtwo{0 & X^*-\ol\zeta_2}{X-\zeta_1 & -I}\,,
\]
this block matrix is a linearization of $(X-\zeta_1)^{-1} (X^*-\ol\zeta_2)^{-1}$. 

In the spirit of \cite{CM} and the Bethe-Salpeter equations, we will express the limit of $\tr_N(X-\zeta_1)^{-1} (X^*-\ol\zeta_2)^{-1}$ in terms of the limit of the resolvent of a larger block matrix whose blocks are linear in $(X-\zeta_i)$ for $i=1,2$. The resolvent of linearized matrix is studied by the matrix Dyson equation. In order to construct this matrix we linearize $\alpha (X-\zeta_1)^{-1} (X^*-\ol\zeta_2)^{-1}$ for some parameter $\alpha$. Let
\begin{equation} \label{defL}
L^{\zeta}_\alpha\,:=\, \mtwo{0 & (X-\zeta_2)^*}{X-\zeta_1 & -\alpha}\,,
\end{equation}
which, by the same logic, is a linearization, as the $(1,1)$-entry of the inverse is $\alpha (X-\zeta_1)^{-1} (X^*-\ol\zeta_2)^{-1}$.

As the matrix $L^{\zeta}_\alpha$ is non-Hermitian it can be sensitive to perturbations. Therefore, we consider its Hermitized version
\begin{equation} \label{defH}
\bs{H}^{\zeta}_\alpha\,:=\, \mtwo{0 & L^{\zeta}_\alpha}{(L^{\zeta}_\alpha)^* & 0}\,,
\end{equation}
as well as the resolvent 
\begin{equation}\label{defG}
\bs{G}^{\zeta}_\alpha(z)\,:=\,(\bs{H}^{\zeta}_\alpha-\1 z \2\bs{1})^{-1}\,,
\end{equation}
with $z= E+ \ii \eta$, where $\eta$ is a non-negative number. Here and in what follows we use boldface to denote $4N \times 4N$ matrices.

For some sufficiently small $\kappa$ to be chosen later, depending only on the model parameters, we consider the indicator function
\[ \Psi_\zeta := \bs{1}\{ \spec |\bs{H}_\alpha^\zeta| \subset [\kappa \Delta_{\zeta}^2 /2, \infty) \} .\] 

We will study the behavior of $\bs{G}^{\zeta}_\alpha(z)$ in the following neighborhood of zero:
\[ \Upsilon_{\zeta} :=  \{ (\alpha,z) :  |\alpha|, |z| <  \kappa \Delta_\zeta^2   \} .\]

The first technical lemma approximates  $\alpha (X-\zeta_1)^{-1} (X^*-\ol\zeta_2)^{-1}$ by the $(3,1)$-entry of $\bs{G}^{\zeta}_\alpha(z)$ when $\Psi_\zeta = 1$.
This lemma will be proven in Section \ref{sec:linear}.

Let $\bs{E}_i \in \C^{4N \times N}$ for $i=1,2,3,4$ be the $4 \times 1$ vector of $N \times N$ matrices with the identity matrix in the $i$-th component. For the following lemma we introduce $x\lesssim y$ to mean there exist a constant $C$, depending only on the model parameters, such that $x \leq C y$.

\begin{lemma} \label{lem:linear}
Let $X$ satisfy Assumption \ref{assum:randmat}, $\zeta \in \Xi_{\infty}$, $|\alpha| >0$. 
Then we have   
\[ \left|  \tr_N\left( \frac{1}{\alpha} \bs{E}_3^*\bs{G}^{\zeta}_\alpha(0)\1\bs{E}_1  - (X-\zeta_1)^{-1}(X^*-\ol\zeta_2)^{-1} \right) \right| \Psi_{\zeta} \\ \lesssim \frac{ \alpha^2}{  \Delta_{\zeta}^{4} }  .\] 
\end{lemma}

To study the resolvent $\bs{G}^{\zeta}_\alpha(z)$ we introduce the matrix Dyson equation (MDE):
\begin{equation} \label{eq:MDE} -\bs{M}^{\zeta}_\alpha(z)^{-1}  = z \bs{1} + \bs{A}^{\zeta} + \bs{\alpha} + \bs{\cal{S}}[\bs{M}^{\zeta}_\alpha(z)] ,
\end{equation}
where the self-energy operator $\bs{\cal{S}} : \C^{4N \times 4N} \to \C^{4N \times 4N} $ is given by
\[
\bs{\cal{S}}[\bs{R}]
\,=\,
\mfour
{ \scr{S}^*[R_{44}] &0 & \scr{S}^*[R_{42}] &  0}
{0 & \scr{S}[R_{33}] &0 & \scr{S}[R_{31}]}
{\scr{S}^*[R_{24}]&0& \scr{S}^*[R_{22} ] &0 }
{0  &\scr{S} [R_{13}]  & 0  &\scr{S}[ R_{11}]}\,,
\]
for any $4 \times 4$ block matrix $\bs{R} \in \C^{4N \times 4N}$ with $R_{ij} \in \C^{N \times N}$ in its $(i,j)$ block,
and
\[ \bs{A}^{\zeta} = \mfour
{0 & 0 & 0 &  \ol{\zeta_2} }
{ 0 & 0 & {\zeta_1} &0}
{0 & \ol{\zeta_1} & 0 & 0}
{ \zeta_2 & 0& 0 & 0},\quad \bs{\alpha} = \alpha( \bs{E}_{24} +\bs{E}_{42}).\]

In the following lemma we use the local law from \cite{AEKN} to show that the resolvent, $\bs{G}_\alpha^\zeta(z)$, is close to the solution of the MDE. We will prove this lemma in Section \ref{sec:locallaw} as well as give a lower bound on the spectrum of $|\bs{H}_{\alpha}^{\zeta}|$, showing $\Psi_\zeta=1$ with high probability.

\begin{lemma}\label{lem:locallaw}
There exist universal constants $\delta>0, p\in \N$ such that 
\[\P\left(  \left|  \tr_N( \bs{E}_3^* \bs{M}^{\zeta}_\alpha(z)  \bs{E}_1 )  -  \tr_N(  \bs{E}_3^* \bs{G}^{\zeta}_\alpha(z)  \bs{E}_1  )  \right|  \le \frac{N^{\epsilon} }{N \Delta_\zeta^p  }   \; : \;  \forall \zeta \in \Xi_\delta  , (\alpha,z) \in \Upsilon_\zeta \right)  \ge 1- \frac{C_{\epsilon,D}}{N^D} \]
for any $\epsilon,D>0$ and some constant $C_{\epsilon,D}>0$.
\end{lemma}

In Section \ref{sec:solat0} we solve the MDE at $\alpha= z= 0$ and then use this solution to show the equation is stable at this point. This stability is then used to show the solution is analytic in $\alpha$ and that the self-consistent density of states (defined below in \eqref{eq: scdos}) associated to the MDE is bounded away from zero. In Section \ref{sec:solnear0} we prove the following lemma and give properties of the solution to the MDE in a neighborhood of zero.
\begin{lemma} \label{lem:alphaexp}
Let $\zeta \in \Xi_{\infty}$ and $(\alpha,0) \in \Upsilon_\zeta$, then
\[ \Big|\tr_N\Big( \frac{1}{\alpha}   \bs{E}_3^* \bs{M}_{\alpha}^\zeta(0)   \bs{E}_1 - \d_{\alpha}|_{\alpha=0} \bs{E}_3^*  \bs{M}_{\alpha}^{\zeta}(0)  \bs{E}_1  \Big) \Big| \lesssim\frac{ |\alpha|}{\Delta_\zeta^2 }  .\]
\end{lemma}

With this regularity we expand the matrix Dyson equation in a series in $\alpha$ and express the order $\alpha$ terms through the order $1$ terms, leading to an equation for $\d_{\alpha}|_{\alpha=0} \tr_N \bs{E}_{3}^* \bs{M}_{\alpha}^{\zeta}(0)  \bs{E}_1 $  which we explicitly solve in Section \ref{sec:compderiv}, yielding the following formula: 

 \begin{lemma} \label{lem:derivalpha}
Let $\zeta \in \Xi_{\infty}$, then
\[ \d_{\alpha}|_{\alpha=0} \tr_N \bs{E}_{3}^* \bs{M}_{\alpha}^{\zeta}(0) \bs{E}_{1}=  \langle (\zeta_1 \ol\zeta_2 - S )^{-1} \mathbbm{1} \rangle.\]

\end{lemma}

We conclude this section by putting together Lemmas \ref{lem:linear} through \ref{lem:derivalpha}  to prove Theorem~\ref{thm:resolventconv}. 

\begin{proof}[Proof of Theorem \ref{thm:resolventconv}]

Using the triangle inequality  
\begin{align*}
  &| \tr_N(X-\zeta_1)^{-1}(X^*-\ol\zeta_2)^{-1} - \langle (\zeta_1 \ol\zeta_2 - S )^{-1}\mathbbm{1} \rangle  | \\
 \leq& | \tr_N(X-\zeta_1)^{-1}(X^*-\ol\zeta_2)^{-1}  -  \alpha^{-1} \tr_N \bs{E}_3^*\bs{G}^{\zeta}_\alpha(0)\1\bs{E}_1  | \\
 &+| \alpha^{-1}\tr_N \bs{E}_3^*\bs{G}^{\zeta}_\alpha(0)\1\bs{E}_1 
 -  \alpha^{-1} \tr_N \bs{E}_3^* \bs{M}_{\alpha}^\zeta(0)   \bs{E}_1    | \\ 
 &+ | \alpha^{-1} \tr_N  \bs{E}_3^* \bs{M}_{\alpha}^\zeta(0)   \bs{E}_1  -\langle (\zeta_1 \ol\zeta_2 - S )^{-1}\mathbbm{1}\rangle  |  .
\end{align*}
By Lemma \ref{lem:linear}, for all $\zeta \in \Xi_{\delta}$
\[ \Big| \tr_N(X-\zeta_1)^{-1}(X^*-\ol\zeta_2)^{-1}  -  \alpha^{-1} \tr_N \bs{E}_3^*\bs{G}^{\zeta}_\alpha(0)\1\bs{E}_1  \Big| \Psi_{\zeta} \lesssim  \frac{|\alpha|^2}{\Delta_{\zeta}^{4}}. \]
In Lemma \ref{lemma:noeigs} we will show
\[
  \P\big(\Psi_{\zeta}=1 \;  : \; \forall \zeta \in \Xi_\delta  \big)\ge 1- \frac{C_D}{N^D}
\]
for any $D$ with some constant $C_D$. 
By Lemma \ref{lem:locallaw}, for every $\zeta \in \Xi_\delta$ we have   
\[| \alpha^{-1}\tr_N \bs{E}_3^*\bs{G}^{\zeta}_\alpha(0)\1\bs{E}_1 -  \alpha^{-1} \tr_N \bs{E}_3^* \bs{M}_{\alpha}^\zeta(0)   \bs{E}_1    |  \lesssim  |\alpha|^{-1} N^{-1+\epsilon} \Delta_\zeta^{-p  } \] 
with probability $1- O(N^{-D})$ for any $D,\epsilon>0$.

By Lemmas \ref{lem:alphaexp} and \ref{lem:derivalpha},   
\[| \alpha^{-1} \tr_N  \bs{E}_3^* \bs{M}_{\alpha}^\zeta(0)   \bs{E}_1  -\langle (\zeta_1 \ol\zeta_2 - S )^{-1}\mathbbm{1} \rangle  |  \lesssim |\alpha|/\Delta_\zeta^2 \] 

Choosing $\alpha= N^{-c'}$ with a sufficiently small positive $c'>0$ completes the proof.

\end{proof}

\section{Matrix Dyson Equation}
\label{sec:MDE}

Recall the matrix Dyson equation \eqref{eq:MDE}:
 \[  -\bs{M}^{\zeta}_\alpha(z)^{-1}  = z \bs{1} + \bs{A}^{\zeta} + \bs{\alpha} + \bs{\cal{S}}[\bs{M}^{\zeta}_\alpha(z)] . \]

For $z \in \C^+$, $\zeta \in \C^2$ and $\alpha \in \R$, let $\bs{M}^{\zeta}_{\alpha}(z)$ be the solution to the equation \eqref{eq:MDE} with positive definite imaginary part, $\Im(\bs{M}^{\zeta}_{\alpha}(z)) =  \frac{1}{2\ii}(\bs{M}^{\zeta}_{\alpha}(z)-\bs{M}^{\zeta}_{\alpha}(z)^*) > 0$. In \cite{Helton01012007} it is shown that when $\Im(z \bs{1} + \bs{A}^{\zeta} + \bs{\alpha}) > 0 $ there is a unique such solution.

To the solution of the MDE we associate an $N$-vector of positive semidefinite matrix-valued measures, $v_j^{z,\alpha} \in \C^{4 \times 4}$, on the real line by
\[ m^{\zeta,\alpha}_{j}(z) = \int_{\R} \frac{ v_j^{\zeta,\alpha}(d \lambda) }{ \lambda - z }, \quad j=1,\ldots, N ,  \] 
where $m^{\zeta,\alpha}_{j}$ is the $4 \times 4$ matrix whose $(k,l)$-entry is the $(j,j)$-entry of the $(k,l)$-block of $\bs{M}_{\alpha}^{\zeta}(z)$.  
In \cite{AEKN}, Proposition 3.2, the authors show the existence of these measures and that they are compactly supported. 

We define the {\it self-consistent density of states} to be the scalar-valued measure
\begin{equation} \label{eq: scdos}
\rho^{\zeta,\alpha}(d \tau) := \frac{1}{4 N} \sum_{j=1}^N \tr v_j^{\zeta,\alpha}(d \tau) .
\end{equation}

The distance between two solutions to the MDE with nearby parameters can be bounded by showing that the MDE is stable. This stability is quantified by the norm of the inverse of the {\it stability operator }
\begin{equation} \label{eq:stab}
 \bs{\cal{L}}_\alpha  :=  \bs{1} - \bs{M}_{\alpha}^{\zeta}(z)\bs{\cal{S}}[\cdot]\bs{M}_{\alpha}^{\zeta}(z).
 \end{equation} 
 In \cite{AEKN} this operator was used to prove the local law and control the smoothness of the solution to the MDE in its parameters. Note that $\bs{\cal{L}}_\alpha$ depends on $\zeta$ and $z$ but we omit this dependence from the notation.

\subsection{Exact solution at 0}
\label{sec:solat0}

At $\alpha = 0, z = 0, |\zeta_1|,|\zeta_2|> 1$ we define \[ \bs{ M}_{0}^{\zeta}(0) := - (\bs{A}^{\zeta})^{-1} ,\] 
which solves \eqref{eq:MDE}. In this section we bound the stability operator associated to this solution. We will use this stability and the implicit function theorem in the next section (Lemma \ref{lem:nearzero}) to control the solution to the MDE in a neighborhood of zero. This shows that $\bs{ M}_{0}^{\zeta}(0)$ in fact agrees with the extension to the real line of the unique solution with positive imaginary part to \eqref{eq:MDE}. Alternatively, the analysis given in \cite{AEK-circ,CHNR} could also be used to show this is the correct solution in the limit as $\Im(z) \to 0$. 

We now introduce some notation and conventions. Products and fractions of vectors are taken componentwise, e.g. $(vw)_i = v_i w_i$. Given an $N$-dimensional vector $v$, we define $\diag(v)$ to be  the $N\times N$ matrix with $v$ placed along the diagonal.  Given a $4N \times 4N$ matrix $\bs{L}$, we define the linear operator acting on $4N \times 4N$ matrices $\bs{\cal{C}}_{\bs{L}}[\bs{R}] := \bs{L}\bs{R}\bs{L} $. Given an $N\times N$ matrix, $T$, and $1 \leq k,l \leq 4$, we define $\bs{E}_{kl}[T]$ to be the $4N \times 4N$ block matrix with $T$ in its $(k,l)$-block and zeros elsewhere. We use the shorthand $\bs{E}_{kl}=\bs{E}_{kl}[I]$. The norm, $\|\cdot \|$, when applied to matrices will denote the usual operator norm induced by the Euclidean metric and when applied to operators acting on matrices will denote the induced operator norm. The norm $\|\cdot \|_{sp}$ on operators acting on matrices will denote the norm induced by the Hilbert-Schmidt norm on matrices. 

\begin{lemma} \label{stabatzero}
Let $\zeta \in \Xi_{\infty}$, then the inverse of the stability operator, $\bs{\cal{ L}}_0=( \bs{1} -\bs{\cal{C}}_{\bs{ M}^{\zeta}_0(0)}\bs{\cal{S}})$, at $z=0$ satisfies the bound

\[\| \bs{\cal{ L}}_0^{-1}\| 
 \lesssim \Delta_\zeta^{-1}. \]

\end{lemma}

\begin{proof}
Substituting the solution $\bs{ M}^{\zeta}_0(0)= -(\bs{A}^{\zeta})^{-1}$ gives
\[
\bs{\cal{C}}_{\bs{M}^{\zeta}_0(0)}\bs{\cal{S}}[\bs{R}] \,=\, 
\mfour
{\frac{1}{\abs{\zeta_2}^2}\scr{S}[R_{11}]  & 0 &\frac{1}{\ol\zeta_1 \zeta_2}\scr{S}[R_{13}]& 0}
{ 0 & \frac{1}{\abs{\zeta_1}^2}\scr{S}^*[R_{22}]  & 0 &\frac{1}{\ol\zeta_1 \zeta_2}\scr{S}^*[R_{24}] }
{\frac{1}{\zeta_1 \ol\zeta_2}\scr{S}[R_{31}]& 0 & \frac{1}{\abs{\zeta_1}^2}\scr{S}[R_{33}] & 0}
{0 & \frac{1}{\zeta_1 \ol\zeta_2}\scr{S}^*[R_{42}] & 0 & \frac{1}{\abs{\zeta_2}^2}\scr{S}^*[R_{44}]}\,.
\]
We define the unitary matrix
\[
\bs{{U}} \,:=\, 
\mfour
{0 & 0 & 0 & \frac{\ol\zeta_2}{\abs{\zeta_2}} }
{0 & 0 & \frac{\zeta_1}{\abs{\zeta_1}} & 0}
{ 0 & \frac{\ol\zeta_1}{\abs{\zeta_1}} & 0 & 0}
{ \frac{\zeta_2}{\abs{\zeta_2}} & 0 & 0 & 0}\,
\]
and the transformation matrix 
\[
\bs{V}\,:=\,
\mfour
{\frac{1}{\abs{\zeta_2}}\diag\sqrt{\frac{v_l}{v_r}} & 0 & 0 & 0}
{ 0 & \frac{1}{\abs{\zeta_1}}\diag\sqrt{\frac{v_r}{v_l}}& 0 & 0}
{0 & 0 & \frac{1}{\abs{\zeta_1}}\diag\sqrt{\frac{v_l}{v_r}} & 0}
{ 0 & 0 & 0&\frac{1}{\abs{\zeta_2}}\diag\sqrt{\frac{v_r}{v_l}}}\,.
\]

The following identities are easy to verify:
\[
\bs{\cal{C}}_{\bs{ M}^{\zeta}_0(0)}\,=\, \bs{\cal{C}}_{\!\sqrt{\bs{V}}}\,\bs{\cal{C}}_{\bs{U}}\2\cal{\bs{\cal{C}}}_{\!\sqrt{\bs{V}}}\,,\qquad 
 \bs{1} -\bs{\cal{C}}_{\bs{ M}^{\zeta}_0(0)}\bs{\cal{S}}\,=\, \bs{\cal{C}}_{\!\sqrt{\bs{V}}}\,( \bs{1} -\bs{\cal{C}}_{\bs{U}}\2\bs{\cal{F}})\bs{\cal{C}}_{\!\sqrt{\bs{V}}}^{-1}\,,
\]
where we introduce the operator 
\[
\bs{\cal{F}}\,:=\,\ \bs{\cal{C}}_{\!\sqrt{\bs{V}}} \2\bs{\cal{S}} \2 \bs{\cal{C}}_{\!\sqrt{\bs{V}}}\,.
\]

\begin{remark}
The operator $\bs{\cal{F}}$ first appeared in \cite{AjankiCorrelated}. In the current context we use it only in the limit $\Im(z) \to 0$.
\end{remark}

By direct calculation one sees that the eigenvalues with largest magnitude for the non-trivial irreducible components of the self-adjoint operator $\bs{\cal{F}}$ are $\pm \frac{1}{\abs{\zeta_i}\abs{\zeta_j}}$, for $i,j$ each equaling 1 or 2, with corresponding eigenmatrices 
\[ \bs{E}_{11}[\diag\sqrt{v_l v_r}] \pm \bs{E}_{44}[\diag\sqrt{v_l v_r}], 
\quad  \bs{E}_{22}[\diag\sqrt{v_l v_r}] \pm \bs{E}_{33}[\diag\sqrt{v_l v_r}], \] 
\[ \bs{E}_{13}[\diag\sqrt{v_l v_r}] \pm \bs{E}_{42}[\diag\sqrt{v_l v_r}],
  \quad  \bs{E}_{24} [\diag\sqrt{v_l v_r}] \pm \bs{E}_{31}[\diag\sqrt{v_l v_r}]. \]

Now using that $\bs{\cal{C}}_{\bs{U}}$ is a unitary operator we have the following bound
\begin{align*} 
&\norm{ ( \bs{1} -\bs{\cal{C}}_{\bs{ M}^{\zeta}_0(0)}\bs{\cal{S}})^{-1} }_{sp}\,\le\, \norm{\bs{V}}\norm{\bs{V}^{-1}} (1-\norm{\bs{\cal{F}}}_{sp})^{-1}\,\\
&\leq \, \frac{\max\{\abs{\zeta_1}, \abs{\zeta_2}\}}{\min\{\abs{\zeta_1}, \abs{\zeta_2}\}}\max\cbb{\frac{v_l}{v_r}, \frac{v_r}{v_l}}\frac{\min\{\abs{\zeta_1}^2,\abs{\zeta_2}^2\}}{\min\{\abs{\zeta_1}^2,\abs{\zeta_2}^2\}-1}\,,
\end{align*}
which is $O(\Delta_{\zeta}^{-1})$. This bound in spectral norm can immediately be extended to a bound in the norm induced by the operator norm on matrices by using inequality (4.55) in \cite{AjankiCorrelated}.

\end{proof}
Note that the final estimate used the uniform lower bound on the left and right Perron-Frobenius eigenvectors of $S$.

\subsection{Solution near 0}
\label{sec:solnear0}

In the following lemma we show the solution to the MDE can be analytically extended from the upper half-plane to a neighborhood of zero. We will continue to denote this extension by $\bs{M}_{\zeta}^{\alpha}(z)$.

\begin{lemma} \label{lem:nearzero}
Let $\zeta \in \Xi_{\infty}$. The solution $\bs{M}_{\zeta}^{\alpha}(z)$ to the MDE, \eqref{eq:MDE}, with $\alpha \in \R$ and $\Im(z) > 0$, has an analytic extension in $\alpha$ and $z$ to the entire region $\Upsilon_\zeta$. 
Moreover, for  $(\alpha,z)\in \Upsilon_\zeta$ the following hold:  
\begin{enumerate}
\item[(1)]  The extension  $\bs{M}^{\zeta}_{\alpha}(z)$ satisfies the MDE, \eqref{eq:MDE}, and the bound 
\begin{equation} \label{eq:contofM} \| \bs{M}^{\zeta}_0(0) - {\bs{M}}^{\zeta'}_\alpha(z)  \| \lesssim (|\alpha| + |z|+|\zeta - \zeta'|) \Delta_{\zeta}^{-1}   , \end{equation}
for any $\zeta' \in \Xi_{\infty}$ such that $|\zeta - \zeta'| \leq \kappa \Delta_{\zeta}^{2}.$
\item[(2)]  The inverse of the stability operator satisfies the bound 
\begin{equation} \label{eq:stabop}  \|   ( \bs{1} -\bs{\cal{C}}_{\bs{ M}^{\zeta}_\alpha(z)}\bs{\cal{S}})^{-1} \| \lesssim \Delta_{\zeta}^{-1}. \end{equation}
\item[(3)]  When $\alpha$ and $z$ are real, the solution $ {\bs{M}}^{\zeta}_\alpha(z)$ is real.
\end{enumerate}
\end{lemma}
Before proving this lemma we show that Lemma \ref{lem:alphaexp} follows as an immediate corollary.
\begin{proof}[Proof of Lemma \ref{lem:alphaexp}]
Lemma \ref{lem:nearzero} shows that $\bs{M}^{\zeta}_\alpha(0)$ is analytic in $\alpha$ and has a Taylor expansion in $\Upsilon_\zeta$:
\[ \bs{M}^{\zeta}_\alpha(0) = \bs{M}_0^{\zeta}(0)+\alpha \d_{\alpha} |_{\alpha=0} \bs{M}_{\alpha}^{\zeta}(0) + O(\alpha^2\Delta_{\zeta}^{-2}).\]
 The claim then follows from $\left( \bs{M}_{0}^\zeta(0)\right)_{31} = 0$.
\end{proof}

\begin{proof}[Proof of Lemma \ref{lem:nearzero}]

We apply the implicit function theorem (Lemma \ref{IFT}) with $\bs{\cal{T}}:\cal D \to \C^{4N\times4N}$ defined by
\begin{equation}  \bs{\cal{T}}(\bs{\widetilde M},\bs D) = ( \bs{\cal{S}}[\bs{\widetilde M}] + \bs{A}^{\zeta}+ \bs D)\bs{\widetilde M} + \bs I ,   \end{equation}
where $\cal D$ is the subset of $\C^{4N\times4N} \times \C^{4N\times4N}$ defined by
\[ \cal D: = \{ (\bs{\widetilde M}, \bs{D}):  \| \bs{\widetilde M} - \bs{M}^{\zeta}_0(0)\|\le c_M \Delta_\zeta, \| \bs{D}\|\le c_D \Delta_\zeta^2\} \]
for some $c_M, c_D$ sufficiently small.

The derivative in the first coordinate in the direction $\bs R$ is:
\begin{equation} \label{eq:deriv1} \nabla^{(1)}_R \bs{\cal{T}}(\bs{\widetilde M},\bs D) =  ( \bs{\cal{S}}[\bs{ \widetilde M}]+ \bs{A}^{\zeta} + \bs D)\bs  R + ( \bs{\cal{S}}[\bs R])\bs{ \widetilde M}.\end{equation}
At $\bs{ \widetilde M}=\bs{M}^{\zeta}_0(0),\bs D=\bs 0$ this simplifies to
\begin{equation}  \label{eq:deriv1Mo} \nabla^{(1)}_R \bs{\cal{T}}(\bs{M}^{\zeta}_0(0) ,\bs 0) =  -\bs{M}^{\zeta}_0(0)^{-1} (\bs{1}-\bs{\cal{C}}_{\bs{M}^{\zeta}_0(0)}\bs{\cal{S}}[\bs R]). \end{equation}
Combining \eqref{eq:deriv1}  with \eqref{eq:deriv1Mo} gives 
\[ \nabla^{(1)}_R \bs{\cal{T}}(\bs{ \widetilde M},\bs D) =  \bs{M}^{\zeta}_0(0)^{-1} (1-\bs{\cal{C}}_{\bs{M}^{\zeta}_0(0)}\bs{\cal{S}}[\bs R])  + ( \bs{\cal{S}}[\bs{ \widetilde M}-\bs{M}^{\zeta}_0(0)] + \bs D)  \bs R + \bs{\cal{S}}[\bs R] (\bs{ \widetilde M}-\bs{M}^{\zeta}_0(0)) . \]

The derivative in the second coordinate in the direction $\bs R$ is 
\[ \nabla^{(2)}_R\bs{\cal{T}}(\bs{ \widetilde M},\bs D) =  \bs R \bs{ \widetilde M} . \]

Thus the operator in the first part of \eqref{eq:IFT} is  
\begin{equation}\label{IFT1} (\bs{1}-\bs{\cal{C}}_{\bs{M}^{\zeta}_0(0)}\bs{\cal{S}})^{-1}\bs{M}^{\zeta}_0(0)\left(( \bs{\cal{S}}[\bs{ \widetilde M}-\bs{M}^{\zeta}_0(0)] + \bs D)  \cdot + \bs{\cal{S}}[\cdot] (\bs{ \widetilde M}-\bs{M}^{\zeta}_0(0))   \right)  \end{equation}
where $\cdot$ indicates the insertion of the $4N \times 4N$ matrix argument. 
 
Using the bounds on $\bs{\widetilde M} $,  $\bs D$ from the definition of $\cal{D}$ along with the bounds $\|\bs{M}^{\zeta}_0(0)\| \leq \max\{1/|\zeta_1|,1/|\zeta_2|\} \leq 1$ and $\| \cal{L}_0^{-1} \| \lesssim \Delta_\zeta^{-1}$ we conclude the norm of operator \eqref{IFT1} is $ \lesssim \Delta_{\zeta}^{-1} (\Delta_\zeta+ \Delta_\zeta^2)$, and the implied constant can be made sufficiently small with appropriate choice of $c_M$ and $c_D$.

The second inequality of \eqref{eq:IFT} is satisfied as the norms of operators are bounded by
\[ 2 c_D \Delta^2_\zeta \|  \bs{M}^{\zeta}_0(0) (\bs{1}-\bs{\cal{C}}_{\bs{M}^{\zeta}_0(0)}\bs{\cal{S}}) ^{-1}    \|  \|  \bs{ \widetilde M}   \|    \lesssim  \Delta_\zeta^2 (\Delta_{\zeta}^{-1})( 1+ \Delta_\zeta)  \lesssim  \Delta_\zeta. \]
Once again the implied constant can be made sufficiently small with appropriate choice of $c_D$.
Thus by the implicit function theorem we conclude for each $\bs{D} = z \bs{I} + \bs{\alpha}+ \bs{A}^\zeta-\bs{A}^{\zeta'}$ there is a unique $\bs{\widetilde M}^{\zeta'}_\alpha(z)$ solving \eqref{eq:MDE} in a neighborhood $\bs{M}_0^\zeta(0)$. Additionally, the solution depends analytically on $\bs{D}$ and satisfies the bound 
\begin{equation} \label{eq:contMtil}
\| \bs{M}^{\zeta}_0(0) - \bs{\widetilde M}^{\zeta'}_\alpha(z) \| \lesssim  (|z|+|\alpha|+| \zeta - \zeta'| )\Delta_{\zeta}^{-1} . \end{equation}
Below we will show that $\bs{\widetilde M}^{\zeta}_\alpha(z)$ has positive imaginary part when $\Im(z+ \bs{A}^{\zeta}+ \bs{\alpha}) >0$, which by uniqueness of the solution to the MDE implies $\bs{\widetilde M}^{\zeta}_\alpha(z) = \bs{ M}^{\zeta}_\alpha(z)$ and concludes the proof of part (1).

Let \[ \widetilde{ \bs{\cal{L}}}_\alpha :=\bs{1} - \bs{M}_{\alpha}^{\zeta}(z)\bs{\cal{S}}[\cdot]\bs{M}_{\alpha}^{\zeta}(z) = \bs{1} - \bs{\cal{C}}_{{\bs{\widetilde M}^{\zeta}_\alpha(z)}}\bs{\cal{S}}  \] be the stability operator, associated to $\bs{\widetilde M}^{\zeta}_\alpha(z) $. We now use \eqref{eq:contMtil} to estimate its inverse in $\cal{D}$ by perturbation theory. Taking the inverse of 
\begin{align*} &\widetilde{ \bs{\cal{L}}}_\alpha = \left( \bs{1} - 
( \bs{\cal{C}}_{{\bs{\widetilde M}^{\zeta}_\alpha(z)}}\bs{\cal{S}} - \bs{\cal{C}}_{\bs{M}^{\zeta}_0(0)}\bs{\cal{S}} )
 \bs{\cal{L}}_0^{-1}\right) ( \bs{1} - \bs{\cal{C}}_{\bs{M}^{\zeta}_0(0)}\bs{\cal{S}}) 
\end{align*}

and choosing $c_M$ and $c_D$ such that $\| \bs{M}^{\zeta}_0(0) - \bs{M}^{\zeta}_\alpha(z) \|$ is sufficiently small gives the bound:
\begin{equation} \label{eq:Mtilstab}
\| \widetilde{ \bs{\cal{L}}}_\alpha^{-1}\| \lesssim \Delta_{\zeta}^{-1}
\end{equation}
which will prove part (2) once we verify $\bs{\widetilde M}^{\zeta}_\alpha(z) = \bs{ M}^{\zeta}_\alpha(z)$ in $
\Upsilon_\zeta$.

We now show the derivative along the imaginary axis of $\bs{ M}^{\zeta}_0(0) $ is positive, implying $\Im( \bs{\widetilde M}^{\zeta}_\alpha(\ii \eta) )$ is positive on the imaginary axis in a neighborhood of $0$ and thus agrees with $\bs{M}^{\zeta}_\alpha(z) $ in this region and then the upper complex plane by analytic continuation. Taking the derivative in $\eta$ of \eqref{eq:MDE} gives:

\[ \partial_\eta|_{\eta=0} \bs{\widetilde  M}^{\zeta}_0(\ii \eta) =   \bs{ M}^{\zeta}_0(0)  \bs{\cal{S}} \left[ \partial_\eta|_{\eta=0} \bs{\widetilde  M}^{\zeta}_0(\ii \eta) \right] \bs{  M}^{\zeta}_0(0) + \ii \left( \bs{\widetilde  M}^{\zeta}_0(0)  \right)^2 \]
then using the formula for $ \bs{ M}^{\zeta}_0(0)$ verifies the claim.

For the remainder of the proof we use $\bs{M}^{\zeta}_\alpha(z)$ to denote the analytic extension of the solution to the MDE to $\Upsilon_\zeta$.

To prove part (3) we follow the argument in the appendix of \cite{AEKN}. We reproduce it here using our bounds on the stability operator.
From the MDE we see that the difference $\bs{M}^{\zeta}_\alpha(0)  -   \bs{M}^{\zeta}_\alpha(z) $ satisfies: 
\begin{align} \label{Delta} 
\bs{\cal{L}}_0[\bs{M}^{\zeta}_0(0)  -   \bs{M}^{\zeta}_\alpha(z) ] =&    \bs{M}^{\zeta}_0(0)(z \bs{I} +\bs{\alpha})    \bs{M}^{\zeta}_0(0) + \frac{1}{2}\left(\bs{M}^{\zeta}_0(0) \bs{\cal{S}}[\bs{M}^{\zeta}_0(0) -   \bs{M}^{\zeta}_\alpha(z) ](\bs{M}^{\zeta}_0(0)  -   \bs{M}^{\zeta}_\alpha(z) )\right. \nonumber \\
& \left.+ (\bs{M}^{\zeta}_0(0)  -   \bs{M}^{\zeta}_\alpha(z) ) \bs{\cal{S}}[\bs{M}^{\zeta}_0(0)  -   \bs{M}^{\zeta}_\alpha(z) ] \bs{M}^{\zeta}_0(0)   \right)   
\end{align}
since the operator $\bs{\cal{L}}_0$ is invertible. Applying $\bs{\cal{L}}_{0}^{-1}$ to \eqref{Delta} and using the implicit function theorem, as above but applied to the subspace of self-adjoint matrices, shows $\bs{M}^{\zeta}_0(0)  -   \bs{M}^{\zeta}_\alpha(z)  = (\bs{M}^{\zeta}_0(0)  -   \bs{M}^{\zeta}_\alpha(z) )^*$, or $\Im(\bs{M}^{\zeta}_\alpha(z) )=0$ since $\Im (\bs{M}^{\zeta}_0(0) )=0$.
\end{proof}

\begin{corollary} \label{lem:gapineig}
Let $\alpha$ be such that $ (\alpha,0) \in \Upsilon_\zeta$, then

\[ \dist(\supp(\rho^{\alpha,\zeta}),0) > \kappa \Delta_\zeta^2 .\]
\end{corollary}

\begin{proof}

By Lemma \ref{lem:nearzero}, part (3), the imaginary part of $\bs{M}^{\zeta}_\alpha(z)$ is zero when $(\alpha, z) \in \Upsilon_\zeta$. Therefore at any such $z$, the self-consistent density of states, $\rho^{\alpha,\zeta}$, is zero and we conclude the desired bound. 
\end{proof}

\subsection{Computation of $ \d_{\alpha} \bs{M}_\alpha^{\zeta}(0)$  at $\alpha=0$}
\label{sec:compderiv}

We now prove Lemma \ref{lem:derivalpha}. To compute $\d_{\alpha} |_{\alpha=0} \bs{M}_{\alpha}^{\zeta}(0)$ in terms of $\bs{M}_0^{\zeta}(0)$ we differentiate \eqref{eq:MDE} in $\alpha$.

\begin{proof}[Proof of Lemma \ref{lem:derivalpha}]

Since this proof is performed at $ z = 0$ we omit the $z$-dependence from our notation.

By differentiating \eqref{eq:MDE} with respect to $\alpha$, then setting $\alpha$ equal to zero and rearranging we arrive at
\begin{equation}\label{eq:expforM1}   \d_{\alpha}|_{\alpha=0}  \bs{M}_{\alpha}^{\zeta} =  \bs{\cal{L}}_0^{-1}(\bs{M}^{\zeta}_0(\bs{E}_{24} +\bs{E}_{42})\bs{M}^{\zeta}_0 ) . \end{equation}

By substituting $\bs{M}^{\zeta}_0 = -(\bs{A}^{\zeta} )^{-1}$ and using that only the $(2,4)$-entry is mapped by $\bs{\cal{L}}_0^{-1}\bs{\cal{C}}_{\bs{ M}^{\zeta}_0}$ to the $(3,1)$-entry, a short calculation gives the expression 
\[   \d_{\alpha}|_{\alpha=0} \bs{E}^*_3 \bs{M}_{\alpha}^{\zeta} \bs{E}_1 =  \left( \zeta_1 \ol\zeta_2 - \scr{S} \right)^{-1} I = \diag( ( \zeta_1 \ol\zeta_2 - S)^{-1} \mathbbm{1} ) \]
and taking the trace leads to the desired expression.
\end{proof}

Recall $S[v_r]=v_r$ and $S^*[v_l]=v_l$ are left and right eigenvectors for $S$. Furthermore define the spectral projection corresponding to the complement of the Perron-Frobenius eigenvalue.
\[
Q[r]\,:=\, r- \frac{\scalar{v_l}{r}}{\scalar{v_l}{v_r}}\,v_r\,.
\]
Then we further expand
\begin{equation}\label{derivdecomp}
( \zeta_1 \ol\zeta_2 - S)^{-1}  \mathbbm{1}
\,=\,
\frac{1}{\zeta_1 \ol\zeta_2  -1}\frac{\avg{v_l} }{\scalar{v_l}{v_r}}\, v_r+
Q ( \zeta_1 \ol\zeta_2  -S )^{-1}Q \mathbbm{1} \,.
\end{equation}

In particular, if the variance profile is constant, we have,
\[  \tr_N( \d_{\alpha}|_{\alpha=0}   \bs{E}^*_3  \bs{M}_{\alpha}^{\zeta} \bs{E}_1 ) = \frac{1}{\zeta_1 \ol\zeta_2 - 1  } .\]

\subsection{Local Law} 
\label{sec:locallaw}
We now use the local law from \cite{AEKN} to show that when $ \dist(\supp(\rho^{\alpha,\zeta}),0)$ is bounded away from zero, $\bs{M}_{\alpha}^\zeta(z)$ is indeed a good approximation of $\bs{G}_{\alpha}^\zeta(z)$. We first show that Corollary \ref{lem:gapineig} can be combined with \cite{AEKN} to show there are no eigenvalues of $\bs{H}_\alpha^{\zeta} $ near zero.

\begin{lemma} \label{lemma:noeigs}
There is a universal constant $\delta>0$ such that  
\[\P( \spec |\bs{H}_\alpha^{\zeta}| \subset [\kappa \Delta_{\zeta}^2 /2,\infty)  \; : \; \forall \zeta \in \Xi_{\delta}, (\alpha,0) \in \Upsilon_\zeta )> 1 - \frac{C_D}{N^D}\] 
for any $D$ and some positive constant $C_D$.
\end{lemma}

\begin{proof}
For any fixed  $\zeta\in \Xi_{\delta}$ and  $\alpha$ such that $ (\alpha,0) \in \Upsilon_\zeta$, Corollary \ref{lem:gapineig} gives a universal constant such that $\supp \rho^{\alpha,\zeta}$ is contained in the interval  $[\kappa \Delta_{\zeta}^2, \infty) $. Combining this with an appropriate choice of $\delta$ in $\Xi_\delta$, Theorem 4.7 in \cite{AEKN} gives
\[\P\Big( \spec |\bs{H}_\alpha^{\zeta}| \subset [\kappa \Delta_{\zeta}^2 /2,\infty)  \Big)> 1 - \frac{C_D}{N^D}.\]

This argument holds for every element of a
regular grid  $\Gamma $ with polynomial size in $N$ in the set of $(\zeta,\alpha)$ such that 
 $\zeta \in \Xi_{\delta}$, $(\alpha,0) \in \Upsilon_\zeta$.
Thus by a simple union bound we have
 \[\P\Big( \spec |\bs{H}_\alpha^{\zeta}| \subset [\kappa \Delta_{\zeta}^2 /2,\infty)  \; : \; (\zeta,\alpha)\in \Gamma
   \Big)> 1 - \frac{C_D}{N^D}.\] 
On this grid, $\Gamma$, the lower bound on the spectrum of $\bs{H}_\alpha^{\zeta}$ implies the norm of the resolvent, $\bs{G}_{\alpha}^{\zeta} $, is $O(\Delta_{\zeta}^{-2})$ with high probability. 

Finally to extend the bound to arbitrary ${\alpha},{\zeta}$, we take the nearest $({\alpha'},{\zeta'})\in \Gamma$. Since $\bs{G}_{\alpha}^{\zeta}(0) = (1 - (\bs{\alpha}'-\bs{\alpha}+\bs{A}^{\zeta'}-\bs{A}^{\zeta}) \bs{G}_{\alpha'}^{\zeta'}(0) )^{-1} \bs{G}_{\alpha'}^{\zeta'}(0) $, choosing the grid to be sufficiently fine, we have the bound $\| \bs{G}_{\alpha}^{\zeta}(0) \|  \lesssim \Delta_{\zeta}^{-2}$ whenever this bound holds on the grid. Using that $\spec |\bs{H}_\alpha^{\zeta}| \subset [\|\bs{G}_{\alpha}^{\zeta}(0)\|^{-1}, \infty)$ concludes the proof.
\end{proof}

\begin{proof}[Proof of Lemma \ref{lem:locallaw}]
For any fixed  $\zeta\in \Xi_{\delta}$ and  $\alpha$ such that $ (\alpha,0) \in \Upsilon_\zeta$
from the local law Lemma B.1 (ii) of \cite{AEKN} there exists a $p>0$ such that 
\[\P\left(  \left|  \tr_N( \bs{E}_3^* \bs{M}^{\zeta}_\alpha(z)  \bs{E}_1 )  -  \tr_N(  \bs{E}_3^* \bs{G}^{\zeta}_\alpha(z)  \bs{E}_1  )  \right|  \le  \frac{N^{\epsilon} }{N \Delta_\zeta^p  }  \right)  \ge 1- \frac{C_{\epsilon,D}}{N^D}. \]
To extend the proof to all  $\zeta\in \Xi_{\delta}$ and  $\alpha$ such that $ (\alpha,0) \in \Upsilon_\zeta$ we follow the same grid argument as in Lemma \ref{lemma:noeigs}. The necessary Lipschitz continuity of $\tr_N( \bs{E}_3^* \bs{M}^{\zeta}_\alpha(z)  \bs{E}_1 ) $ was established in Lemma \ref{lem:nearzero} part (1).
\end{proof}
We also have the following corollary of Lemma \ref{lemma:noeigs}.
\begin{corollary} \label{corr:resolventbound}
There is a universal constant $\delta>0$ and a constant $C>0$ such that  
\[\P\Big(  \big\| (X-\zeta_1)^{-1}  (X^*-\ol{\zeta}_2)^{-1}\big\| \le  C N^{4\delta}  \; : \; \forall \zeta \in \Xi_{\delta}  \Big)> 1 - \frac{C_D}{N^D},\] 
 for any $D$ and some positive constant $C_D$. 
\end{corollary}

\subsection{Bound on Linearization}
\label{sec:linear}

\begin{proof}[Proof of Lemma \ref{lem:linear}]

From the block structure of $\bs{H}^{\zeta}_\alpha$ we deduce the following block structure of $\bs{G}^{\zeta}_\alpha(\ii \eta)$
\[
\bs{G}^{\zeta}_\alpha(\ii \eta)\,=\, 
\mtwo{-\ii\eta & L^{\zeta}_\alpha}{(L^{\zeta}_\alpha)^* & -\ii\eta}^{-1}\,=\, \mtwo{\ii\eta(\eta^2+LL^*)^{-1} & (\eta^2+LL^*)^{-1}L}{L^*(\eta^2+LL^*)^{-1} & \ii\eta(\eta^2+L^*L)^{-1}}\,.
\]
We use the $\eta>0$ regularization of the inverse to ensure all inverses exist in our derivation. In the final estimate we will invoke the assumption that $ \spec{|\bs{H}_{\alpha}^{\zeta} |} \subset [\kappa \Delta_{\zeta}^2 /2, \infty)$ and safely set $\eta=0$.

We compute the $(3,1)$-block of $\bs{G}$ through
\bes{
&\bs{E}_3^*\bs{G}^{\zeta}_\alpha(\ii \eta)\1\bs{E}_1\\
\,&=\, 
E_1^* \mtwo{0 & (X-\zeta_1)^*}{X- \zeta_2 & -\alpha}\mtwo{\eta^2+(X-\zeta_2)^*(X-\zeta_2)& -\alpha (X-\zeta_2)^*}{-\alpha (X-\zeta_2) & \eta^2+\alpha^2+(X-\zeta_1)(X-{\zeta_1})^*}^{-1}E_1 \\
\,&=\,
\alpha (X-\zeta_1)^* \left(\eta^2+\alpha^2+(X-\zeta_1) (X-\zeta_1)^* \right)^{-1} 
(X- \zeta_2) \left( \eta^2+(X-\zeta_2)^* J (X-\zeta_2) \right)^{-1}, \\
}
where \[J = \frac{\eta^2+(X-\zeta_1)(X-\zeta_1)^*}{\eta^2+\alpha^2+(X-\zeta_1)(X-\zeta_1)^*}\]
and \[
E_1\,=\, \vtwo{1}{0} \in \C^{2n\times n}\,.
\]

To estimate the difference when $\alpha > 0$ is small and to justify setting $\eta=0$ we use the estimates:
\[\| (X-\zeta_i)^{-1} \| \leq \dist(\spec |\bs{H}_{0}^{\zeta}|,0)^{-1},  \]
\[J-I= \alpha^2( \eta^2+\alpha^2+(X-\zeta_1)(X^*-\ol\zeta_1) )^{-1} 
\leq \alpha^2 \dist(\spec |\bs{H}_{0}^{\zeta}|,0)^{-2} . \]
Bounding the normalized trace by the operator norm and applying the above estimates 
\begin{align*}
& (X-\zeta_1)^{-1}(X^*-\ol\zeta_2)^{-1} -\bs{E}_3^*\bs{G}^{\zeta}_\alpha(0)\1\bs{E}_1/\alpha   \\
&=(X-\zeta_1)^{-1}  (J - I) J^{-1}  (X^*-\ol\zeta_2)^{-1}   \\
&+ (X-\zeta_1)^{-1}   \alpha^2  \left( \alpha^2+(X-\zeta_1)(X^*-\ol\zeta_1) \right)^{-1} 
 J^{-1} (X^*-\ol\zeta_2)^{-1}   
\end{align*}
yields the desired bound for the resolvent on the event $ \spec |\bs{H}_\alpha^\zeta |\subset [\kappa \Delta_{\zeta}^2 /2, \infty)$.

\end{proof}

\section{Proof of Theorem \ref{thm:heat} and \ref{thm:stationary}}
\label{sec:mainresult}

In this section we consider differential equations
\[
\partial_t u_t \,=\, -u_t + gX u_t
\]
with $0 < g \leq 1$ and initial value $u_0$ distributed uniformly on the sphere, and 

\[
\dd  u_t \,=\, (-u_t + g X u_t) \dd t + \dd B_t
\]
where for each $t$, $B_t$ is a vector of standard Brownian motions, independent of $X$.

In the first equation we consider the behavior of the squared norm in the large $N$ limit, when averaged over initial conditions
\begin{equation} \label{time}
\E_{u_0}\2\norm{u_t}^2\,=\, \tr_N[\ee^{t(gX^*-I)}\ee^{t(gX-I)}].
\end{equation}
The case $g=1$ is of particular interest, as the damping term and coupling term are balanced.

Before continuing, we briefly consider the behavior when $g=1$ of such an expression if $X$ is replaced with $W$, a Wigner matrix, meaning $W$ is Hermitian and the entries on and above the diagonal are i.i.d. random variables with variance $1/(4N)$. Wigner's semicircle law asserts that the empirical spectral distribution is supported on $[-1,1]$ with density $( 2 /\pi) \sqrt{1-\abs{x}^2}$. 

In this case, as $W = W^*$ the limiting analogous expression for \eqref{time} can be expressed as an integral against the limit of the empirical density. At large times, the behavior decays like $t^{-3/2}$. Indeed, by the semicircle law, as $t \to \infty$, we have: 
\begin{align*}
&\lim_{N \to \infty} \tr[\ee^{t(W-I)}\ee^{t(W-I)}] = \frac{2}{\pi} \int_{-1}^{1} \ee^{2 t(x-1)}\sqrt{1-\abs{x}^2}\dd x\\
&\,=\, \sqrt{2}  \frac{2}{\pi} \int_{0}^{\infty} \ee^{-2tx}x^{1/2}\dd x + O(t^{-5/2}) \,=\, \frac{1}{2 \sqrt{\pi}} \frac{1}{t^{3/2}} + O(t^{-5/2}),
\end{align*}
where in the second line we do the change of variables $x \to 1-x$ and gain the error term by extending the bounds of integration to infinity.
 
From this calculation we see the asymptotic behavior of $u_t$ in the Hermitian case is governed by the behavior of its density at the edge. For random Wigner-type matrices with a variance profile the square root behavior at the edge is universal \cite{AjankiQVE}. Thus the $t^{-3/2}$ scaling is universal in the Hermitian case.

We now return to the non-Hermitian setting.
\begin{proof}[Proof of Theorem \ref{thm:heat}]
By Cauchy's theorem, 
\bes{ \tr_N[\ee^{t(gX^*-I)}\ee^{t(gX-I)}]  = \left( \frac{1}{2 \pi \ii} \right)^2 \oint_\gamma  \dd \zeta_1 \oint_{\ol{\gamma}} \dd \ol{\zeta}_2 \, \ee^{t(g\zeta_1+g\ol\zeta_2-2)} \tr_N[(X-\zeta_1)^{-1}(X^*-\ol\zeta_2)^{-1}]  ,}
where $\gamma$ is a circle that enclosed all the eigenvalues of $X$, traversed clockwise and $\ol{\gamma}$ is the same circle traversed counterclockwise. 

We apply Theorem \ref{thm:resolventconv} and then the decomposition in \eqref{derivdecomp}, computing each term separately. For $ t \leq N^\delta $ we have
\bes{ \tr_N[\ee^{t(gX^*-I)}\ee^{t(gX-I)}]  =  & \left( \frac{1}{2 \pi \ii} \right)^2  \oint_{\gamma} \dd \zeta_1\oint_{\ol{\gamma} }  \dd \ol{\zeta_2} \, \left[  \frac{\avg{v_l}\avg{v_r}}{\scalar{v_l}{v_r}}   \ee^{t  (g\zeta_1+g\ol\zeta_2-2)} \frac{1}{\zeta_1 \ol\zeta_2  -1}\,  \right. \\
& \left.+  \ee^{t g (\zeta_1+\ol\zeta_2-2)} 
\langle Q ( \zeta_1 \ol\zeta_2  -S )^{-1}Q \mathbbm{1} \rangle \right] + \epsilon_N\,.  }
where $|\epsilon_N | < N^{-\xi} $ with probability $1 - N^{-D}$ for any $D$.

The first term in the integral is:
\bes{
&\frac{1}{(2 \pi \ii)^2} \oint_{\gamma} \dd \zeta_1\oint_{\ol{\gamma} }  \dd \ol{\zeta_2} \, 
\,\ee^{t(g\zeta_1+g\ol\zeta_2-2)}\frac{1}{\zeta_1 \ol\zeta_2-1} \\ 
\,&=\frac{1}{(2 \pi \ii)^2} \oint_{\gamma} \dd \zeta_1\oint_{\ol{\gamma} }  \dd \ol{\zeta_2} \, 
\,\ee^{t(g\zeta_1+g\ol\zeta_2-2)}\frac{\zeta_1^{-1}}{\ol\zeta_2 -\zeta_1^{-1}}  \\ 
\,&=\, 
\frac{1}{2 \pi \ii} \oint_{\abs{\zeta_1}=1} \dd \zeta_1  
\,\ee^{t(g\zeta_1^{-1}+g\zeta_1-2)} \,\frac{1}{\zeta_1}  \\
\,&=\,
\frac{1}{2 \pi \ii}  \oint_{\abs{\zeta_1}=1} \dd \zeta_1 \frac{1}{\zeta_1}  
\,\ee^{2t(g\Re{\zeta_1}-1)}   \\
\,&=\,
J_0(2 g \ii t) \ee^{-2 t}
}
where $J_0$ is the Bessel function of the first kind.

For the second term in the integral we will use Lemma \ref{lem:Sbound} to deform the contour into the region with modulus less than one:

\bes{
 \,&\, \left| \frac{1}{(2 \pi \ii)^2} \oint_{\gamma} \dd \zeta_1\oint_{\ol{\gamma} }  \dd \ol{\zeta_2} \,
\,\ee^{t(g\zeta_1+g\ol\zeta_2-2)} \langle Q ( \zeta_1 \ol\zeta_2  -S )^{-1}Q \mathbbm{1} \rangle   \right| \\ 
\,&=\, 
\left| \frac{1}{(2 \pi \ii)^2} \oint_{|\zeta_1| = 1-\epsilon} \dd \zeta_1\oint_{|\ol{\zeta}_2| = 1-\epsilon }  \dd \ol{\zeta_2} \,
\,\ee^{t(g\zeta_1+g\ol\zeta_2-2)} \langle Q ( \zeta_1 \ol\zeta_2  -S )^{-1}Q \mathbbm{1} \rangle   \right| \\ 
\,&\lesssim \ee^{-2(1-g)t}\ee^{-cgt}, 
}
where we have used that all the poles of the integrand lie in the disk centered at the origin of radius $1- 2 \epsilon$ and that $| \ee^{t(g\zeta_1+g\ol\zeta_2-2)} | \lesssim \ee^{-2(1-g -g\epsilon) t}$, where $\epsilon$ is obtained in Lemma \ref{lem:Sbound}.

\end{proof}

We now turn to the differential equation in stationarity and compute the autocorrelation function.
\begin{proof}[Proof of Theorem \ref{thm:stationary}]

The following computation of $R(0)$ shows that the solution \eqref{ODEstat} has finite variance with respect to $B_t$ for large $N$ with high probability taken with respect to $X$. 

The autocorrelation of this solution is
\bes{ 
R(\tau) &=   \E_{B}[\langle x_t, x_{t+\tau}  \rangle ] \\
&= \int_{-\infty}^t   \tr_N(\ee^{(-1+g X)(t+\tau-s)} 
 \ee^{(-1+g X^*)(t-s)}       ) \, \dd s \\
&=\int_0^\infty \ee^{-2 u - \tau} \tr_N(\ee^{gX(u+\tau)} \ee^{gX^*u} ) \, \dd u   
}
which by Theorem \ref{analyticfunc} and Cauchy's theorem is 
\[ R(\tau) = \left( \frac{1}{2 \pi \ii} \right)^2 \frac{\avg{v_l}\avg{v_r}}{\scalar{v_l}{v_r}}  \int_0^\infty \ee^{-2 u - \tau}  
 \oint_{\gamma} \dd \zeta_1\oint_{\ol{\gamma} }  \dd \ol{\zeta_2} \,
\ee^{g\zeta_1(u+\tau)} \ee^{g\ol{\zeta_2} u} \langle (\zeta_1 \ol{\zeta_2} - S)^{-1} \mathbbm{1} \rangle  \dd u  + \epsilon_N \]
where $|\epsilon_N| \leq N^{-\xi} $ for $\tau \leq N^{-\delta}$.

Once again we apply the decomposition in \eqref{derivdecomp}. The first term is
\bes{
&   \int_0^\infty \ee^{-2 u - \tau} \frac{1}{(2 \pi \ii)^2} \oint_{\gamma} \dd \zeta_1\oint_{\ol{\gamma} }  \dd \ol{\zeta_2} \, 
\,\ee^{g\zeta_1(u+\tau)} \ee^{g\ol{\zeta_2}u}  \frac{1}{\zeta_1 \ol\zeta_2-1} \dd u \\ 
& =  \int_0^\infty  \frac{1}{(2 \pi \ii)} \oint_{\gamma} \dd \zeta_1 \, 
\,\ee^{g\zeta_1(u+\tau)+ g \zeta_1^{-1}u -2 u - \tau }  \zeta_1^{-1} \dd u \\ 
& =  \frac{1}{(2 \pi \ii)} \oint_{\gamma} \dd \zeta_1 \, 
\,\frac{1}{g\zeta_1^2 -2 \zeta_1 +g} \ee^{g\zeta_1\tau - \tau} \,  \dd u ,\\ 
}
and by factoring the denominator
\[  g\zeta_1^2 -2 \zeta_1 +g = g \left(\zeta_1 - \frac{1 + \sqrt{1 -  g^2 } }{ g}\right)\left(\zeta_1 - \frac{1 - \sqrt{1 -  g^2 } }{ g}\right)  \]
we have 
\bes{ \frac{1}{(2 \pi \ii)} \oint_{\gamma} \dd \zeta_1 \, 
\,\frac{1}{g\zeta_1^2 -2 \zeta_1 +g} \ee^{g\zeta_1\tau - \tau} \,  \dd u   =   
\,\frac{1}{ 2\sqrt{1 -  g^2 } } \ee^{  - \left(\sqrt{1 -  g^2 }  \right)\tau }  . \\ 
}

The subleading term is bounded as before, giving faster exponential decay.  
\end{proof}

\appendix 
\section{Inverse Lemma}

\begin{lemma} \label{lem:Sbound} 
Let $S$ satisfy Assumption \ref{assum:randmat}. Let $v_l$ and $v_r$ denote the left and right eigenvectors respectively. Then $Q(S-z)^{-1}Q$ has bounded norm for $z$ larger than $1-\epsilon$ for some small $\epsilon$, depending only on the model parameters.

\end{lemma}

\begin{proof}

We will now prove that there exists $\eps \sim 1$ such that
\[
\sup_{\abs{z} \ge 1-2\eps}\normb{(QSQ-z)^{-1}} \,\lesssim\, 1\,.
\]
In fact we will prove that 
\bels{resolvent bound}{
\sup \{\normb{(S-z)^{-1}}:\; z \not \in D_{1-2\eps}(0) \cup D_{\eps}(1)\} \,\lesssim\, C_\epsilon\,. 
}
for all sufficiently small $\epsilon > 0$. In particular, $\epsilon$ can be chosen so that $\spec(S) \cap D_\epsilon(1) = \{1 \}$.
First we extend the resolvent to include the resolvent of $S^*$ (which may only increase the bound). Then we symmetrize
\[
\mtwo{S & 0}{  0 & S^*}- z \mtwo{1 & 0}{  0 & 1}\,=\, T^{-1} \pbb{\mtwo{0 & 1}{1 & 0}
\mtwo{0 & A^*}{A& 0}-z \mtwo{1 & 0}{  0 & 1}}T\,,
\]
where 
\[
T \,:=\, \mtwo{\sqrt{\frac{v_l}{v_r}}& 0}{0 & \sqrt{\frac{v_r}{v_l}}} \,, \qquad  A\,:=\, \sqrt{\frac{v_l}{v_r}}S\sqrt{\frac{v_r}{v_l}}\,.
\]
Since $\norm{T}\sim\norm{T^{-1}}\sim 1$ we have
\[
\normb{(S-z)^{-1}}\,\lesssim\,  \normbb{\pbb{\mtwo{0 & 1}{1 & 0}
\mtwo{0 & A^*}{A& 0}-z \mtwo{1 & 0}{  0 & 1}}^{-1}}\,.
\]
Note that with $x=\sqrt{v_lv_r}$ we have $Ax=A^*x=x$. In particular, 
\[
\avg{(x,x),(x,-x)}^\perp=\avg{(x,0),(0,x)}^\perp=\{(q_1,q_2): q_1,q_2 \perp x \}
\]
is invariant under both
\[
\mtwo{0 & 1}{1 & 0} \qquad \text{and} \qquad \mtwo{0 & A^*}{A& 0}\,.
\]
We conclude that 
\[
\mtwo{0 & 1}{1 & 0}
\mtwo{0 & A^*}{A& 0}
\,=\, \vtwo{x}{x} \otimes  \vtwo{x}{x} + \vtwo{x}{-x}\otimes \vtwo{x}{-x} + P  \mtwo{0 & 1}{1 & 0}P
\mtwo{0 & A^*}{A& 0}P\,,
\]
with $P$ the orthogonal projection onto $\avg{(x,0),(0,x)}^\perp$. Using the spectral gap of the matrix
\[
\spec \mtwo{0 & A^*}{A& 0} \,=\, \{-1\} \cup[-1+3\eps,1-3\eps] \cup\{1\}\,,
\]
(see, for instance, Lemma 5.6 in \cite{AjankiCPAM} applied to $(A^*A)^K$) for some $\eps \sim 1$ implies \eqref{resolvent bound}. Here we used the uniform primitivity of $S^*S$ and that $v_r\sim v_l\sim 1$. 

\end{proof}

\section{Implicit Function Theorem}

\begin{lemma}[IFT] \label{IFT}
Suppose $\C^A$ and $\C^B$ are equipped with norms which are both denoted by $\|\mspace{1mu}\cdot \mspace{1mu}\|$, and let the linear operators mapping between these spaces be equipped by the induced operator norms. Let $\eps_A, \eps_B>0$, $a_0\in \C^A$, $b_0 \in \C^B$ and
 $T:B_{\eps_A}^A(a_0) \times B_{\eps_B}^B(b_0)\to \C^A$ be a continuously differentiable function with invertible derivative $\nabla^AT(a_0,b_0)$ with respect to the first argument at the origin and $T(a_0,b_0)=0$. Suppose that 
 \begin{equation} \label{eq:IFT}
\sup_{a,b} \| \mathrm{Id}_{\C^A}-(\nabla^AT(a_0,b_0))^{-1}\nabla^AT(a,b)  \|\,\le\, \frac{1}{2}\,,\quad 
2\mspace{1mu}\eps_B\|(\nabla^AT(a_0,b_0))^{-1} \| \sup_{a,b}\|\nabla^BT(a,b)\|\,\le\, \eps_A\,,
 \end{equation}
 where the supremum is taken over $(a,b) \in B_{\eps_A}^A(a_0) \times B_{\eps_B}^B(b_0)$.\\
Then there is a unique function $f: B_{\eps_B}^B(b_0) \to B_{\eps_A}^A(a_0)$ such that
\[
T(f(b),b)\,=\, 0\,, \qquad b \in B_{\eps_B}^B(b_0)\,.
\]
The function $f$ is continuously differentiable. If $T$ is analytic then so is $f$.
\end{lemma}

\bibliographystyle{abbrv}
\bibliography{2pt}

\end{document}